\documentclass[12pt]{article}
\usepackage{latexsym, amsmath, amsfonts, amsthm, amssymb}
\usepackage{times}
\usepackage{a4wide}
\usepackage{mathrsfs}

\def \rank {{ \rm rank }}
\def \Vol {{ \rm Vol }}

\def \RR {\mathbb R}

\def \EE {\mathbb E}

\def \ZZ {\mathbb Z}

\def \PP {\mathbb P}

\def \eps {\varepsilon}
\def \vphi {\varphi}

\def \cE {\mathcal E}
\def \cF {\mathcal F}

\def \cA {\mathcal A}

\newtheorem{theorem}{Theorem}[section]

\newtheorem{lemma}[theorem]{Lemma}

\newtheorem{proposition}[theorem]{Proposition}

\theoremstyle{definition}
\newtheorem{remark}[theorem]{Remark}

\def\myffrac#1#2 in #3{\raise 2.6pt\hbox{$#3 #1$}\mkern-1.5mu\raise 0.8pt\hbox{$
		#3/$}\mkern-1.1mu\lower 1.5pt\hbox{$#3 #2$}}

\def\qed{\hfill $\vcenter{\hrule height .3mm
		\hbox {\vrule width .3mm height 2.1mm \kern 2mm \vrule width .3mm
			height 2.1mm} \hrule height .3mm}$ \bigskip}

\begin{document}

\title{Poisson approximation of random lattices}
\date{}
\author{Boaz Klartag}
\maketitle

\abstract{
Fix a subset $S \subset \RR^n$ of volume at most $c n$
that satisfies $S \cap (-S) = \emptyset$. We consider two point processes in $S$: the first is
the Poisson point process of intensity one, and the second is the restriction of
a random lattice  to $S$, where the random lattice is distributed uniformly in the space
of covolume-one lattices.
We  show  that
the total variation distance between these two point processes is
 at most $C e^{-c' n}$, where $c, C, c' > 0$ are universal constants.
 }

\section{Introduction}
\label{sec1}

We revisit the connection between random lattices of covolume one
and the Poisson point process of intensity one in high dimensions.
A {\it simple point process} in $\RR^n$ is a probability distribution over
the space of locally finite subsets of $\RR^n$. We consider two simple point processes: 
\begin{enumerate}
	\item The first is the Poisson point process of intensity one in $\RR^n$. Thus $P \subseteq \RR^n$ is a random countable subset, such that for any fixed subset $S \subseteq \RR^n$ with $0 < \Vol_n(S) < \infty$, the random variable
	$$ |P \cap S| $$
	is a Poisson random variable with parameter $\lambda = \Vol_n(S)$, i.e., $\PP( |P \cap S|  = k ) = e^{-\lambda} \lambda^k / k!$. Here $|X| = \#(X)$ is the cardinality of the set $X$.
	
\medskip An important property  is that the random variables $|P \cap A_1|,\ldots, |P \cap A_N|$ are independent random variables for any pairwise disjoint sets $A_1,\ldots, A_N \subseteq \RR^n$ of finite volume.  Moreover, for $m \geq 1$, when conditioning on the event that $|P \cap S| = m$, the random subset $P \cap S$ has the same distribution as the set $\{X_1,\ldots,X_m \}$,
where $X_1,\ldots,X_m$ are independent random vectors, all distributed uniformly in $S$.
 We refer the reader e.g. to Last and Penrose \cite{LP} for information about Poisson processes.
	
\item A lattice of full rank in $\RR^n$ is a set of the form $T(\ZZ^n)$ where $T: \RR^n \rightarrow \RR^n$ is an invertible linear map.
The covolume of the lattice is $|\det(T)|$.
	Write $\mathscr{X}_n$ for the space of all  lattices of covolume one in $\RR^n$.
	The space $\mathscr{X}_n$ is a homogeneous space under the action of the group $SL_n(\RR)$, where the action of $g \in SL_n(\RR)$ on the lattice $L \subseteq \RR^n$
is the lattice $$ g.L = \{ g(x) \, ; \, x \in L \}. $$	
	Minkowski and Siegel \cite{siegel} discovered that there is a unique Haar	{\it probability} measure on $\mathscr{X}_n$ which is invariant under the action of $SL_n(\RR)$.
    Let $L$ be a random lattice of covolume one, distributed uniformly in $\mathscr{X}_n$, i.e., its law is the Haar probability measure.
	The second simple point process that we consider is the random countable subset $L \subseteq \RR^n$. For more information on random lattices we refer the reader e.g. to
Gruber and Lekkerkerker \cite[Section 19.3]{GL} or to Marklof \cite{marklof} and references therein.
\end{enumerate}

A priori, a lattice in $\RR^n$ is a structured object that has little to do with the Poisson process whose points do not influence each other. This intuition is indeed
correct in low dimensions. However, works by
Rogers\footnote{In \cite{rogers3}, Rogers thanks Mr. W. M.  Gorman from the Department of Economics for remarking that Rogers' earlier results about random lattices suggested that the asymptotic distribution might be of the Poisson type.} \cite{rogers3} and Schmidt \cite{schmidt} from the 1950s  suggest that, quite unexpectedly,
in high dimensions the two point processes mentioned above are similar to each other, at least when both are restricted to a subset $S \subseteq \RR^n$ with $S \cap (-S) = \emptyset$
and of volume  at most $c n$. Here $c > 0$ is a universal constant.
The works by Holm \cite{Holm}, Kim \cite{Kim0, Kim} and  S\"odergren \cite{sodergren, sodergren1} described below
analyze several formulations and consequences of this effect. In this paper we propose an additional formulation of the same phenomenon.

\medskip Let $S \subseteq \RR^n$ be a fixed set of finite, non-zero volume.
We endow the collection of all finite subsets of $S$, denoted by $\cF(S)$, with
the minimal topology with respect to which $A \mapsto \sum_{x \in A} \vphi(x)$
is a continuous functional on $\cF(S)$ for any compactly-supported, continuous function
$\vphi: \RR^n \rightarrow \RR$. Observe that the laws
of $P \cap S$ and of $L \cap S$ are Borel probability measures supported on $\cF(S)$.
For a non-empty finite subset $X \subseteq \RR^n$ we write
\begin{equation}  \iota(X) = \left \{ \begin{array}{cl} 1 & \textrm{the vectors in } X  \textrm{ are linearly independent} \\ 0 & \textrm{otherwise} \end{array} \right. \label{eq_1727} \end{equation}
We also set $\iota(\emptyset) = 1$ and note that $\iota(\{ 0 \}) = 0$.
A basic relation
between the Poisson process and the lattice process is the following identity,
whose straightforward proof using Siegel-type summation formulae is given in Section \ref{sec2}.

\begin{proposition} Let $S \subseteq \RR^n$ be a measurable set with $S \cap (-S) = \emptyset$ and $0 < \Vol_n(S) < \infty$. Let $f: \cF(S) \rightarrow \RR$ be a bounded, measurable
	function such that $f(X) = 0$ if $|X| \geq n$. Then,
	\begin{equation}  \EE \sum_{X \subseteq L \cap S}
		f (X) \cdot \iota(X) =
		\EE \sum_{X \subseteq S \cap P }
		f (X),
		\label{eq_1819} \end{equation}
	where the sum runs over all subsets $X$ of $L \cap S$ on the left-hand side, and of $S \cap P$ on the right-hand side. \label{prop_1832}
\end{proposition}

The {\it total variation distance} between the laws of the two random sets $P \cap S$ and $L \cap S$ is defined as
\begin{equation}  d_{TV}(P \cap S, L \cap S) =  \sup_{\cA \subseteq \cF(S)} |\PP (P \cap S \in \cA)
	- \PP(L \cap S \in \cA)| \label{eq_1132} \end{equation}
where the supremum runs over all events $\cA$, i.e., all Borel subsets $\cA \subseteq \cF(S)$. In view of Proposition \ref{prop_1832}, linearly dependent tuples are the main obstacle in showing that the two point processes $L \cap S$ and $P \cap S$ are close to each other. By combining Proposition \ref{prop_1832}
with a variant of Schmidt's sieve  \cite{schmidt} and results by Rogers \cite{rogers2, rogers3} we establish the following:

\begin{theorem} Let $S \subseteq \RR^n$ be a measurable set with $S \cap (-S) = \emptyset$ and $0 < \Vol_n(S) < c n$. Then,
\begin{equation} d_{TV}(L \cap S, P \cap S) \leq C e^{-c' n},
\label{eq_1230} \end{equation}
	where $c, c', C > 0$ are  universal constants.
\label{thm1}	
\end{theorem}

Thus, for any property of finite subsets of $S$, the probability that $L \cap S$
satisfies this property is very close to the probability that $P \cap S$ has the property; the difference
between the two probabilities is at most $C e^{-c' n}$.
Theorem \ref{thm1} is proven in Section \ref{sec4}, while the variant of Schmidt's sieve that
we use is presented in Section~\ref{sec_inc_exc}.

\medskip
Under the assumptions of Theorem \ref{thm1}, Schmidt \cite{schmidt} showed that
the probability that $L \cap S$ is empty is close to the probability that $P \cap S$ is empty, with an additive approximation error of at most $C e^{-c' n}$.
This estimate implies a bound for the sphere packing density in $\RR^n$, and it is also closely
related to the covering density, see Rogers \cite{rogers4}.
More generally, Proposition 3.3 in Kim~\cite{Kim} states that for all $k \leq \tilde{c} n$,
$$ \left| \PP ( |L \cap S| = k ) \, - \, \PP ( |P \cap S| = k ) \right| \leq C e^{-c' n}. $$
Thus, up to the precise values of the universal constants, Theorem \ref{thm1} is formally a generalization of these results.

\medskip We make no effort to optimize the universal constants in Theorem \ref{thm1}. Up to the precise values of these universal constants, the results by S\"odergren \cite{sodergren} and Kim \cite{Kim} about the lengths of the shortest $c n$
vectors in a random lattice may be recovered  from Theorem \ref{thm1}, by using standard analysis of the Poisson process. Indeed, if we let
\begin{equation}  S = \{ x \in \RR^n \, ; \, |x| \leq r_n, x_1 > 0 \}, \label{eq_1102} \end{equation} where $r_n > 0$ is chosen so that $S$ is a half-ball of volume $c n$, then 
almost surely $L \cap S$ contains one lattice vector from each pair $\pm v \in L \setminus \{ 0 \}$ with $|v| \leq r_n$. 
With probability of at least $1 -C e^{-cn}$ there are at least $c' n$ of such lattice vectors, as follows from Theorem \ref{thm1} and standard Poisson tail estimates. Similarly, some
results by S\"odergren \cite{sodergren1}, Kim \cite{Kim} and Holm \cite{Holm} about the angles between the $c n$ shortest vectors, as well as the results regarding the first $c n$ successive minima, should also follow from Theorem \ref{thm1}.

\medskip  The conclusion of Theorem \ref{thm1} is false, in general, when  $\Vol_n(S) \geq n$. Indeed, if $S$ is a half-ball of volume at least $n$  as in (\ref{eq_1102}), then with non-negligible probability, the random set $L \cap S$ contains $n$
vectors such that the matrix whose columns are these vectors has an integer determinant.
The corresponding event for the random set $P \cap S$ occurs with probability zero. Moreover, when
$K \subseteq \RR^n$ is a convex set with the origin in its interior and $S = \{ x \in K \, ; x_1 > 0 \}$, we have\footnote{While discussed as an open problem in \cite{schmidt2}, we note that  (\ref{eq_245}) has been known to experts for a long time.}
\begin{equation}  \EE |L \cap S|^n = +\infty \label{eq_245} \end{equation}
while $\EE |P \cap S|^q < \infty$ for any $q > 0$. To prove (\ref{eq_245}), use
e.g. the bound in Kleinbock and Margulis \cite[Section 7.1]{KM} together with the observation that if $(\eps K) \cap L \neq \{ 0 \}$
then $|L \cap S| \geq c_K /\eps$.

\medskip Our notation is fairly standard. The Euclidean norm of $x = (x_1,\ldots,x_n) \in \RR^n$ is denoted by $|x| = \sqrt{ \sum_i x_i^2}$. We write $\Vol_n$ for $n$-dimensional volume,
and $SL_n(\RR)$ is the group of all $n \times n$ real matrices of determinant one.
All measurable sets and functions are assumed Borel measurable, and an empty sum equals zero.
We write $1_{\cA}(x)$ for the indicator of the set $\cA$, that equals one when $x \in \cA$ and vanishes when $x \not \in \cA$.
We write ${n \choose k}$ for the number of $k$-element subsets of $\{1,\ldots,n\}$,
thus $ { n \choose k } = 0$ when $k > n$.  For a subset $X \subseteq \RR^n$ we write $\rank(X)$ for the dimension of the linear
space spanned by the vectors of $X$.

\medskip 
Unless stated otherwise, we use $C, c, c', C', \tilde{C}, \tilde{c}$
to denote various positive universal constants whose value may change from one line to the next.
All of the universal constants in this paper are effective; that is, we could provide explicit numerical values for them if desired.

\medskip {\it Acknowledgements.} I would like to thank Barak Weiss for his explanations on random lattices and for directing my attention to Schmidt's paper \cite{schmidt}.
Thanks also to Sergey Avvakumov, Seungki Kim, Jens Marklof and Anders S\"odergren for helpful discussions. Supported by a grant
from the Israel Science Foundation (ISF).

\section{Linear independence of the lattice vectors in $S$}
\label{sec2}

Fix a measurable  subset $S \subseteq \RR^n$ with $S \cap (-S) = \emptyset$ and with $0 < \Vol_n(S) < \infty$.
Let $L \subseteq \RR^n$ be a random lattice of covolume one, distributed
uniformly in $\mathscr{X}_n$. Let the random set
$P \subseteq \RR^n$ be the set of atoms of a Poisson point process of intensity one in $\RR^n$. With probability one, both random sets
$$ L \cap S \qquad \text{and} \qquad P \cap S $$
are finite. Recall from (\ref{eq_1727}) that $\iota(X) = 1$ for $X \subseteq \RR^n$
if and only if the vectors in $X$ are linearly independent. Otherwise $\iota(X) = 0$.

\begin{proof}[Proof of Proposition \ref{prop_1832}]
	The sums in (\ref{eq_1819})
	are well-defined since $P \cap S$ and $L \cap S$ are finite sets almost surely.
	Let $\ell=1,\ldots,n-1$ and assume that $f: \cF(S) \rightarrow \RR$ is a bounded, measurable
	function such that $f(A) = 0$ if $|A| \neq \ell$. We will show that under these assumptions,
	both expressions in (\ref{eq_1819}) equal $$ \frac{1}{\ell!} \int_{S^{\ell}} g(x_1,\ldots,x_{\ell}) d x_1 \ldots dx_{\ell} $$
	where $S^{\ell} = S \times \ldots \times S \subseteq (\RR^n)^{\ell}$ and $g: S^{\ell} \rightarrow \RR$
	is defined via $g(x_1,\ldots,x_{\ell}) := f(\{ x_1,\ldots,x_{\ell} \})$
	for distinct $x_1,\ldots,x_{\ell} \in S$. The function $g$ is bounded and hence integrable
	on $S^{\ell}$.
	The fact that
	\begin{equation}  \EE \sum_{\substack{ x_1,\ldots,x_{\ell} \in L \cap S \\ \textrm{linearly independent}}}
		g (x_1,\ldots,x_{\ell}) =  \int_S \ldots \int_S g(x_1,\ldots,x_{\ell}) dx_1\ldots dx_{\ell} \label{eq_1729} \end{equation}
	is proven in Rogers \cite{rogers2} and is mentioned without proof already in Siegel \cite{siegel}. See also Schmidt \cite{schmidt0}.
	One way to prove this fact is to show that the Borel measure
	\begin{equation} \EE \sum_{\substack{x_1,\ldots,x_{\ell} \in L  \\ \textrm{linearly independent}}}
		\delta_{x_1,\ldots,x_{\ell}} \label{eq_2227} \end{equation}
	on $(\RR^n)^{\ell}$ is locally finite and $SL_n(\RR)$-invariant, and therefore a constant multiple of the Lebesgue measure. By considering the measure of large balls, one can prove  that this proportionality constant equals one.

\medskip Let us now analyze the expectation on the right-hand side of (\ref{eq_1819}). Conditioning on the event that
	$|S \cap P| = m$ we know that the set $S \cap P$ is equidistributed with $\{ X_1,\ldots,X_m \}$ where $X_1,X_2,\ldots, X_m$ are independent random vectors distributed uniformly in $S$.
	Therefore, when $m \geq \ell$,
	$$ \EE \left[ \sum_{\substack{x_1,\ldots,x_{\ell} \in S \cap P \\ \textrm{distinct}}}
	g (x_1,\ldots,x_{\ell}) \, \, \Big| \, \, |S \cap P| = m \right]
	= \frac{m!}{(m-\ell)!} \cdot
	\int_{S^{\ell}} g(x_1,\ldots,x_{\ell}) \frac{dx_1 \ldots d x_{\ell}}{\Vol_n(S)^{\ell}}. $$
	Consequently, with $\lambda = \Vol_n(S)$,
	\begin{align} \nonumber \EE \sum_{\substack{x_1,\ldots,x_{\ell} \in S \cap P \\ \textrm{distinct}}}
		g (x_1,\ldots,x_{\ell})  & = \sum_{m = \ell}^{\infty} \frac{e^{-\lambda} \lambda^m}{m!}
		\cdot \frac{m!}{(m-\ell)!} \cdot
		\int_{S^{\ell}} g(x_1,\ldots,x_{\ell}) \frac{dx_1 \ldots d x_{\ell}}{\Vol_n(S)^{\ell}}
		\\ & = \sum_{m = \ell}^{\infty} \frac{e^{-\lambda} \lambda^{m-\ell}}{(m- \ell)!}
		\cdot \int_{S^{\ell}} g(x_1,\ldots,x_{\ell}) dx_1 \ldots d x_{\ell}. \label{eq_1522}
	\end{align}
	The sum in (\ref{eq_1522}) equals one. Therefore, by (\ref{eq_1729}) and (\ref{eq_1522}),
	$$ \EE \sum_{\substack{x_1,\ldots,x_{\ell} \in L \cap S \\ \textrm{linearly independent}}}
	g (x_1,\ldots,x_{\ell}) =
	\EE \sum_{\substack{x_1,\ldots,x_{\ell} \in P \cap S \\ \textrm{distinct}}}
	g (x_1,\ldots,x_{\ell}).
	$$
	Dividing by $\ell!$ we prove (\ref{eq_1819}) in the particular case where
	$f$ satisfies the condition that $f(A) = 0$ when $|A| \neq \ell$. We refer to this condition
	as condition $(\ell)$. For the general case, we note that we may decompose $f = \sum_{\ell=0}^{n-1} f_{\ell}$
	where $f_{\ell}$ is a bounded, measurable function satisfying condition $(\ell)$. We have proved that
	(\ref{eq_1819}) holds true when $f$ is replaced by $f_{\ell}$ for $\ell=1,\ldots,n-1$. The equality (\ref{eq_1819}) holds trivially when $f$
	is replaced by $f_0$. By linearity, we conclude that (\ref{eq_1819}) holds for the function $f$ itself, completing the proof.
\end{proof}

Set $\lambda = \Vol_n(S)$.  As in Schmidt \cite{schmidt}, for a finite set $A \subseteq \RR^n$
let us define $\rho_k^j(A)$ to be the number of
subsets $X \subseteq A$ of cardinality $k$ with $\rank(X) = j$. It follows from  Proposition \ref{prop_1832} that if $0 \leq k \leq n-1$ then
\begin{equation}  \EE \rho_k^k (L \cap S) = \EE \sum_{X \subseteq L \cap S} 1_{\{ \#(X) = k \}} \cdot \iota(X) = \EE \sum_{X \subseteq P \cap S} 1_{\{ \#(X) = k \}} =  \frac{\lambda^k}{k!},
	\label{eq_1450} \end{equation}
where the last passage follows from the fact that $\#(P \cap S)$
is a Poisson random variable with parameter $\lambda$, and from the standard computation  $\sum_{m=k}^{\infty} {m \choose k} \cdot e^{-\lambda} \lambda^m / m!  = \lambda^k / k!$.

\medskip A non-trivial estimate proven in Schmidt \cite[Theorem 3]{schmidt}, relying also on ideas and computations from Rogers \cite{rogers2,rogers3}, states that for $1 \leq k \leq n-1$,
\begin{equation} \EE \rho_k^{k-1} (L \cap S) \leq \frac{\lambda^{k-1}}{(k-1)!} \cdot [3^k (3/4)^{n/2} + 5^k 2^{-n} ]. \label{eq_1451} \end{equation}
The proof of (\ref{eq_1451}) relies on analysis of Siegel-type summation formulae for specific patterns of linear dependence, as well as on delicate counting arguments. 
Observe that in the case where $\lambda \leq c n$ and $k \leq c n$, the right-hand side of (\ref{eq_1451}) is much smaller than that of (\ref{eq_1450}). Thus, in this case, the expected number of linearly independent $k$-tuples in $L \cap S$
greatly exceeds the number of $k$-tuples whose linear rank equals $k-1$. The following proposition was proven by Kim \cite[Proposition 3.1]{Kim}, using the methods of
Rogers and Schmidt.

\begin{proposition}[Kim] Suppose that $\lambda = \Vol_n(S) \leq n/30$. Then
	with probability of at least $1 - C e^{-cn}$, the set $L \cap S$
	consists of at most $n/10$ vectors, and these vectors are linearly independent.
	Here, $c, C > 0$ are universal constants.
	\label{prop_1445}
\end{proposition}

\begin{proof}[Another proof of Proposition \ref{prop_1445} using Schmidt's results] Write
	$d_k(L \cap S)$ for the number of subsets of $L \cap S$ of cardinality $k$ that consist of linearly independent vectors, and are contained in a strictly larger subset of $L \cap S$ whose linear rank is still $k$. A moment of reflection reveals that
	$$ d_{k}(L \cap S) \leq {k+1 \choose k} \cdot \rho_{k+1}^{k}(L \cap S). $$
	Indeed, if we go over all $(k+1)$-element subsets that are counted in the definition of $\rho_{k+1}^{k}(L \cap S)$, and for each of them consider all of its $k$-element subsets, then we cover all of the subsets counted in the definition of $d_k(L \cap S)$. From  (\ref{eq_1451}) we thus see that for $k \leq n-2$,
	\begin{equation}
		\EE  d_k(L \cap S)
		\leq (k+1) \cdot \frac{\lambda^k}{k!} \cdot  (3^{k+1} (3/4)^{n/2} + 5^{k+1} 2^{-n} ).
		\label{eq_1718} \end{equation}
	Recall that  $\lambda \leq n/30$.
	By using the inequalities $k! \geq (k/e)^k$
	and $(e a/x)^x \leq e^a$, we see that for $k \leq n/10$, the right-hand side of (\ref{eq_1718}) is at most
	$$ 10 n \cdot \left( \frac{en}{30 k} \right)^k \cdot  (3^{\frac{n}{10}} (3/4)^{\frac{n}{2}} + 5^{\frac{n}{10}} 2^{-n} )
	\leq 10 n \cdot e^{\frac{n}{30} } \cdot  (3^{\frac{n}{10}} (3/4)^{\frac{n}{2}} + 5^{\frac{n}{10}} 2^{-n} ) \leq C e^{-cn}. $$
	Thus the right-hand side of (\ref{eq_1718}) is at most $C e^{-cn}$ if $k \leq n/10$. Since $d_k(L \cap S)$ is a non-negative integer,
	we conclude  that the probability that the random variable $d_k(L \cap S)$ is positive 
is at most  $C' e^{-c'n}$. 	Therefore,
	\begin{equation}  \PP \left( \exists 1 \leq k \leq n/10, \ d_k(L \cap S) > 0 \right) \leq \sum_{k=1}^{\lfloor n/10 \rfloor} C' e^{-c'n} \leq C e^{-c n}. \label{eq_1111} \end{equation}
	If $d_k(L \cap S) = 0$ for all $k \leq n/10$, then for any subspace $E \subseteq \RR^n$
	of dimension at most $n/10$, the vectors in $L \cap S \cap E$ are linearly independent.
	The probability of this event is at least $1 - C e^{-cn}$ according to (\ref{eq_1111}).

\medskip
	Next, from (\ref{eq_1450}) and the Markov-Chebyshev inequality, the probability that there exist at least $\ell = \lceil n/10 \rceil$ linearly independent vectors in $L \cap S$ is at most
	$$ \EE \rho_\ell^\ell(L \cap S) = \frac{\lambda^\ell}{\ell!} \leq \left( \frac{ e\lambda}{\ell} \right)^\ell \leq C e^{-cn}.
	$$
	In other words, with probability  of at least $1 - C e^{-c n}$, the linear span of the set $L \cap S$ is of dimension at most $n/10$.
	
	\medskip In summary, consider the event where
	the set $L \cap S$
	spans a subspace of dimension at most $n/10$,
	and moreover, for any subspace $E \subseteq \RR^n$ of dimension at most $n/10$, the vectors in $L \cap S \cap E$ are linearly independent.
	By the preceding two paragraphs, the probability of this event is at least $1 - C e^{-cn}$. Under this event, the set $L \cap S$ contains at most $n/10$ vectors, which are linearly independent.
\end{proof}

\section{Schmidt's sieve}
\label{sec_inc_exc}

The usual inclusion-exclusion principle states that for any finite set $B$ and
integers $R_1,R_2\geq 0$, with $R_1$ odd and $R_2$ even,
\begin{equation}
\sum_{j=0}^{R_1}(-1)^j {|B|\choose j}
\leq
1_{\{B=\emptyset\}}
\leq
\sum_{j=0}^{R_2}(-1)^j {|B|\choose j}.
\label{eq_1503} \end{equation}
An analogous, more sophisticated sieve was introduced by Schmidt
\cite{schmidt} and modified by Kim \cite{Kim0, Kim}. A variant of this sieve is presented in detail in this section.

\medskip
The {\it rank deficiency} of a finite set $X\subseteq \RR^n$ is defined as
$$
\delta(X)=|X|-\rank(X)\geq 0.
$$
Thus $\delta(X)=0$ if and only if $X$ consists of linearly independent vectors.
For a set $B$, we write ${B\choose k}$ for the collection of all subsets of
$B$ of cardinality $k$.
For a finite set $B\subseteq \RR^n$ and integers $k,r\geq 0$, define
\begin{equation}
\mathcal S_{k,r}(B)
=
\left\{
X\in {B\choose k+r}\, ;\,
\delta(X)\leq 1_{\{r\textrm{ is odd}\}}
\right\} \label{eq_1423}
\end{equation}
and
\begin{equation}
\mathcal T_{k,r}(B)
=
\left\{
X\in {B\choose k+r}\, ;\,
\delta(X)\leq 1_{\{r\textrm{ is even}\}}
\right\}.
\label{eq_1424} \end{equation}
Thus $\mathcal S_{k,r}(B)$ consists of the linearly independent subsets of
$B$ of cardinality $k+r$ when $r$ is even, and of the subsets of cardinality
$k+r$ and rank deficiency at most one when $r$ is odd. The convention for
$\mathcal T_{k,r}(B)$ is the opposite one.
Recall that $\cF(\RR^n)$ is the collection  of all finite subsets of $\RR^n$.

\begin{proposition}
Let $\vphi:\mathcal F(\RR^n)\to [0,\infty)$ be a function such that
$\vphi(Y)=0$ whenever the vectors in $Y$ are linearly dependent.
Let $B\subseteq \RR^n$ be finite, and let $k\geq 0$. Let $R_1,R_2\geq 0$
be integers such that $R_1$ is odd and $R_2$ is even. Then,
\begin{equation*}
\sum_{r=0}^{R_1}
(-1)^r
\sum_{X\in \mathcal S_{k,r}(B)}
\sum_{Y\in {X\choose k}}
\vphi(Y)
\leq
\vphi(B)1_{\{|B|=k\}}
\leq
\sum_{r=0}^{R_2}
(-1)^r
\sum_{X\in \mathcal T_{k,r}(B)}
\sum_{Y\in {X\choose k}}
\vphi(Y).
\end{equation*}
\label{prop_1725}
\end{proposition}

The remainder of this section is devoted to the proof of Proposition
\ref{prop_1725}. Let $A\subseteq B\subseteq \RR^n$ be finite sets. For $r\geq 0$, set
\begin{equation}
\sigma_r(A,B)
=
\#\left\{
Y\subseteq B\setminus A\, ;\, |Y|=r,\
\delta(A\cup Y)\leq 1_{\{r\textrm{ is odd}\}}
\right\}
\label{eq_1421} \end{equation}
and
\begin{equation}
\tau_r(A,B)
=
\#\left\{
Y\subseteq B\setminus A\, ;\, |Y|=r,\
\delta(A\cup Y)\leq 1_{\{r\textrm{ is even}\}}
\right\}.
\label{eq_1422} \end{equation}
When $r$ is even, the quantity $\sigma_r(A,B)$ counts the number of subsets of size $r$ of $B \setminus A$
such that when we add them to $A$, we obtain a linearly independent set. When $r$ is odd,
in place of linear independence we require rank deficiency at most one. The convention for $\tau_r(A,B)$ is the opposite one.

\begin{lemma}
Let $A\subseteq B\subseteq \RR^n$ be finite sets, and assume that $A$ consists
of linearly independent vectors. Then for every $R_1,R_2\geq 0$ with $R_1$
odd and $R_2$ even,
\begin{equation}
\label{eq_1650}
\sum_{r=0}^{R_1}(-1)^r\sigma_r(A,B)
\leq
1_{\{A=B\}}
\leq
\sum_{r=0}^{R_2}(-1)^r\tau_r(A,B).
\end{equation}
\label{lem_1444}
\end{lemma}

\begin{proof}
Let $C=B\setminus A$ and denote $N=|C|$. If $N=0$  then $A=B$ and for all $r \geq 0$,
$$
\sigma_r(A,B)=\tau_r(A,B)=1_{ \{ r=0 \}}.
$$
Thus both sums in (\ref{eq_1650}) are equal to $1$, and the lemma follows.
From now on we may therefore assume that $A \neq B$ and hence $$ N\geq 1. $$ For $r\geq 0$ define
$$
\mathcal P_r
=
\{Z\subseteq C\, ;\, |Z|=r,\ \delta(A\cup Z)=0\}
$$
and
$$
\mathcal Q_r
=
\{Z\subseteq C\, ;\, |Z|=r,\ \delta(A\cup Z)\leq 1\}.
$$
For $r>N$ we have $\mathcal P_r=\mathcal Q_r=\emptyset$. Moreover, for all
$r\geq 0$,
\begin{equation}
\sigma_r(A,B)
=
\left\{
\begin{array}{ll}
|\mathcal P_r| & r\textrm{ is even}\\
|\mathcal Q_r| & r\textrm{ is odd}
\end{array}
\right.
\qquad
\tau_r(A,B)
=
\left\{
\begin{array}{ll}
|\mathcal Q_r| & r\textrm{ is even}\\
|\mathcal P_r| & r\textrm{ is odd}
\end{array}
\right.
\label{eq_1738} \end{equation}
For $r=0,\ldots,N$, set
\begin{equation}
p_r=\frac{|\mathcal P_r|}{\binom Nr},
\qquad
q_r=\frac{|\mathcal Q_r|}{\binom Nr}.
\label{eq_1739} \end{equation}
Thus $p_r$ is the probability that a
random completion of $A$ to a subset of $B$ of cardinality $|A|+r$ is still
linearly independent, and $q_r$ is the probability that its rank deficiency
is at most one. We also put $p_r=q_r=0$ for $r>N$. Clearly, for all $r \geq 0$,
\begin{equation}
p_r \leq q_r. \label{eq_1414}
\end{equation}
Note that the probability $p_r$ that such a random completion would consist of linearly-independent vectors is non-increasing with $r$,
as well as the probability $q_r$ that its rank deficiency is at most one. Therefore, for any $r \geq 0$,
\begin{equation}
 p_r \geq p_{r+1} \qquad \text{and} \qquad q_r \geq q_{r+1}.
 \label{eq_1702} \end{equation}
Another observation is that if $X \subsetneq B$ is linearly independent, then by adding any element from $B \setminus X$ to $X$, we obtain
a set $Y \subseteq B$ whose rank deficiency is at most one. It thus follows that for $r=0,\ldots,N-1$,
\begin{equation} p_r\leq q_{r+1}.
\label{eq_1703} \end{equation}
For $r \geq 0$ write
\begin{equation}
c_r=\sum_{j=0}^{r}(-1)^j\binom Nj = (-1)^r\binom{N-1}{r},
\label{eq_1239} \end{equation}
with $c_{-1} = 0$, where we recall that  ${N-1 \choose r} = 0$ if $r > N-1$.
By using summation by parts we see that for any numbers $d_0,\ldots,d_R$,
\begin{equation}
\sum_{r=0}^{R}(-1)^r\binom Nr d_r
= \sum_{r=0}^{R} (c_r - c_{r-1}) d_r
= c_Rd_R+\sum_{r=0}^{R-1}c_r(d_r-d_{r+1}).
\label{eq_1241} \end{equation}

\medskip {\it The left-hand side inequality in (\ref{eq_1650}).}
In order to prove the left-hand side inequality in (\ref{eq_1650}) we set $R = R_1$, an odd number. For $r \geq 0$ define
\begin{equation}
d_r=
\left \{ \begin{array}{ll}
p_r & r\textrm{ is even} \\
p_{r-1} & r\textrm{ is odd}
\end{array} \right.
\label{eq_1737} \end{equation}
which is a number between zero and one. Note that if $r$ is odd, necessarily $r-1 \geq 0$ and $d_r = p_{r-1}$ is well-defined.
Note that for $r \geq 0$,
\begin{equation}
\left \{ \begin{array}{ll} d_r - d_{r+1} = 0 & r \textrm{ is even} \\ d_r - d_{r+1} \geq 0 & r \textrm{ is odd} \end{array} \right. \label{eq_1311} \end{equation}
Indeed, if \(r\) is even, then \(d_r - d_{r+1}=p_r - p_r = 0\), while if \(r\) is
odd, then $r \geq 1$ and  $d_r-d_{r+1}=p_{r-1}-p_{r+1}\geq 0$ by (\ref{eq_1702}). Thus
(\ref{eq_1311}) holds true.
Next, we claim that for all $r \geq 0$,
\begin{equation}
(-1)^r\sigma_{r}(A,B)
\leq
(-1)^r\binom Nr d_r.
\label{eq_1736}
\end{equation}
Indeed, if $r > N$ then $\binom Nr = 0$ and both the left-hand side and the right-hand side of (\ref{eq_1736})
vanish, by (\ref{eq_1738}). It thus suffices to prove (\ref{eq_1736}) for $r=0,\ldots,N$. If $r$ is even, then by (\ref{eq_1738}), (\ref{eq_1739})
and (\ref{eq_1737}),
$$ (-1)^r \sigma_{r}(A,B) = (-1)^r \binom Nr p_r = (-1)^r \binom Nr d_r. $$
Hence (\ref{eq_1736}) holds true
when $r$ is even. If $r$ is odd, then $r \geq 1$ and $r-1$ is even.
Thus by (\ref{eq_1738}), (\ref{eq_1739}) and (\ref{eq_1703}),
$$ (-1)^r \sigma_{r}(A,B) = -\binom Nr q_r \leq -\binom Nr p_{r-1} = (-1)^r \binom Nr d_r, $$
where  we used  (\ref{eq_1737}) in the last passage. This completes the proof of (\ref{eq_1736}) in all cases.
We deduce from (\ref{eq_1311}), (\ref{eq_1736}) and the summation by parts (\ref{eq_1241}) that
\begin{align}  \nonumber
\sum_{r=0}^{R}(-1)^r\sigma_{r}(A,B)
& \leq
\sum_{r=0}^{R}(-1)^r\binom Nr d_r =
c_Rd_R+\sum_{r=0}^{R-1}c_r(d_r-d_{r+1})
\\ & \leq \sum_{r=0}^{R-1}c_r(d_r-d_{r+1}) = \sum_{\substack{r=0 \\ r \textrm{ is odd}}}^{R-1}c_r(d_r-d_{r+1}) \leq 0,
\label{eq_1319}
\end{align}
where in the first passage in (\ref{eq_1319}) we used that $c_R d_R \leq 0$ since $d_R \geq 0$ by (\ref{eq_1737}) while
$c_R \leq 0$ by (\ref{eq_1239}) as $R$ is an odd number. In the last passage in (\ref{eq_1319}) we used that
$c_r \leq 0$ by (\ref{eq_1239}) as $r$ is odd, while $d_r - d_{r+1} \geq 0$. Since $A \neq B$, the left-hand side inequality in (\ref{eq_1650}) follows from~(\ref{eq_1319}).

\medskip {\it The right-hand side inequality in (\ref{eq_1650}).}
Set
$R = R_2$, an even number. For $r \geq 0$ define
\begin{equation}
d_r=
\left \{ \begin{array}{ll}
q_r & r\textrm{ is even and } r\leq N \\
q_{r+1} & r\textrm{ is odd and } r \leq N-1 \\
q_N & r \geq N
\end{array} \right.
\label{eq_1737_} \end{equation}
For $r\geq 0$,
\begin{equation}
\left \{ \begin{array}{ll}
c_r(d_r-d_{r+1}) = 0 & r \textrm{ is odd} \\
d_r-d_{r+1} \geq 0 & r \textrm{ is even}
\end{array} \right.
\label{eq_1311_} \end{equation}
Indeed, consider first the case where $r$ is odd. If $r<N$, then
$d_r=d_{r+1}=q_{r+1}$. If $r \geq N$, then $c_r=0$
by (\ref{eq_1239}). Thus
$c_r(d_r-d_{r+1})=0$ whenever $r$ is odd.
Let us now prove (\ref{eq_1311_}) assuming that $r$ is even. If
$r<N-1$, then by (\ref{eq_1702}) and (\ref{eq_1737_}), $$ d_r-d_{r+1}=q_r-q_{r+2}\geq 0. $$ If $r=N-1$, then
$d_r-d_{r+1}=q_{N-1}-q_N\geq 0$
again by (\ref{eq_1702}). If $r \geq N$, then
$d_r-d_{r+1}=q_N-q_N = 0$.
This proves (\ref{eq_1311_}).
Next, let us show that for all $r \geq 0$,
\begin{equation}
(-1)^r\tau_{r}(A,B)
\geq
(-1)^r\binom Nr d_r.
\label{eq_1736_}
\end{equation}
Indeed, if $r>N$ then $\binom Nr=0$ and both the left-hand side and the
right-hand side of (\ref{eq_1736_}) vanish. It remains to prove
(\ref{eq_1736_}) for $r=0,\ldots,N$. If $r$ is even, then by
(\ref{eq_1738}), (\ref{eq_1739}) and (\ref{eq_1737_}),
$$
(-1)^r\tau_r(A,B)
=
(-1)^r\binom Nr q_r
=
(-1)^r\binom Nr d_r.
$$
Hence (\ref{eq_1736_}) holds when $r$ is even. Next, consider the case where $r$ is odd. If $r$ is odd and $r<N$,
then by (\ref{eq_1738}), (\ref{eq_1739}) and (\ref{eq_1703}),
$$
(-1)^r\tau_r(A,B)
=
-\binom Nr p_r
\geq
-\binom Nr q_{r+1}
=
(-1)^r\binom Nr d_r,
$$
where we used (\ref{eq_1737_}) in the last passage. If $r$ is odd and
$r=N$ then
$$
(-1)^r\tau_r(A,B)
=
-p_N \geq -q_N =
(-1)^r\binom N r d_r,
$$
where the inequality $p_N \leq q_N$ follows from (\ref{eq_1414}) above.
This completes the proof of (\ref{eq_1736_}) in all cases.
We deduce from (\ref{eq_1311_}), (\ref{eq_1736_}) and the summation by parts
(\ref{eq_1241}) that
\begin{align}  \nonumber
\sum_{r=0}^{R}(-1)^r\tau_{r}(A,B)
& \geq
\sum_{r=0}^{R}(-1)^r\binom Nr d_r
=
c_Rd_R+\sum_{r=0}^{R-1}c_r(d_r-d_{r+1})
\\
& \geq
\sum_{r=0}^{R-1}c_r(d_r-d_{r+1})
=
\sum_{\substack{r=0 \\ r \textrm{ is even}}}^{R-1}c_r(d_r-d_{r+1})
\geq 0.
\label{eq_1319_}
\end{align}
Here, in the first passage in (\ref{eq_1319_}) we used that $c_Rd_R\geq 0$,
since $d_R\geq 0$ by (\ref{eq_1737_}) while $c_R\geq 0$ by
(\ref{eq_1239}) as $R$ is an even number. In the last passage in (\ref{eq_1319_}) we used that
$c_r\geq 0$ by (\ref{eq_1239}) as $r$ is even, while
$d_r-d_{r+1}\geq 0$ by (\ref{eq_1311_}). The right-hand side inequality in
(\ref{eq_1650}) follows from (\ref{eq_1319_}). The proof of the lemma is
complete.
\end{proof}

\begin{proof}[Proof of Proposition \ref{prop_1725}]
Consider first the sum
$$
L=
\sum_{r=0}^{R_1}
(-1)^r
\sum_{X\in \mathcal S_{k,r}(B)}
\sum_{Y\in {X\choose k}}
\vphi(Y).
$$
Since $\vphi(Y)=0$ whenever the vectors in $Y$ are linearly dependent, only
linearly independent sets $Y\in {B\choose k}$ contribute to the sum $L$. Let us fix a linearly independent subset
$Y \subseteq B$ of cardinality $k$. A term involving $\vphi(Y)$ can only arise from a set
$X\in\mathcal S_{k,r}(B)$ satisfying
$$
X=Y\cup Z,
\qquad
Z\subseteq B\setminus Y,
\qquad
|Z|=r.
$$
Consequently, in view of (\ref{eq_1423}) and (\ref{eq_1421}), the coefficient of $\vphi(Y)$ in the sum $L$ is
$$
\sum_{r=0}^{R_1}(-1)^r\sigma_r(Y,B).
$$
By Lemma \ref{lem_1444}, applied with $A=Y$, this coefficient is at most
$1_{\{Y=B\}}$. Since $\vphi\geq 0$, we obtain
\begin{equation}
L
\leq
\sum_{\substack{Y\in {B\choose k}\\ Y\textrm{ linearly independent}}}
1_{\{Y=B\}}\vphi(Y) = \vphi(B)1_{\{|B|=k\}},
\label{eq_1427} \end{equation}
proving the left-hand side inequality in the conclusion of the proposition. Next, set 
$$
U=
\sum_{r=0}^{R_2}
(-1)^r
\sum_{X\in \mathcal T_{k,r}(B)}
\sum_{Y\in {X\choose k}}
\vphi(Y).
$$
Again, only linearly independent sets $Y\in {B\choose k}$ contribute to the sum. Fix
such a linearly-independent set $Y$. A term in the sum $U$ involving $\vphi(Y)$ arises from a set
$X\in\mathcal T_{k,r}(B)$ such that
$$
X=Y\cup Z,
\qquad
Z\subseteq B\setminus Y,
\qquad
|Z|=r.
$$
Hence the coefficient of $\vphi(Y)$ in the sum $U$ is
$$
\sum_{r=0}^{R_2}(-1)^r\tau_r(Y,B).
$$
By Lemma \ref{lem_1444}, applied with $A=Y$, this coefficient is at least
$1_{\{Y=B\}}$. Since $\vphi\geq 0$, it follows that
$$
U
\geq
\sum_{\substack{Y\in {B\choose k}\\ Y\textrm{ linearly independent}}}
1_{\{Y=B\}}\vphi(Y) = \vphi(B) 1_{\{|B|=k\}}.
$$
This proves the right-hand side inequality in the conclusion of the proposition, and the proof is complete.
\end{proof}

\section{Proof of Theorem \ref{thm1}}
\label{sec4}

Let us set
$$ c_0 = 1/200. $$
Suppose that $n \geq 2$ and fix a measurable subset  $S \subseteq \RR^n$ with $S \cap (-S) = \emptyset$ such that
$$ \lambda = \Vol_n(S) $$
satisfies
$$ 0 < \lambda \leq c_0 n.
$$
Recall that $\cF(S)$ is the  collection of all finite subsets of $S$.

\begin{lemma} Assume that $n \geq 20$, let
$0\leq k\leq n/10$, and let $\cA_k\subseteq \cF(S)$ be a Borel set containing
only subsets of cardinality exactly $k$. Then
$$
\PP(L\cap S\in \cA_k)-\PP(P\cap S\in \cA_k)\leq C e^{-cn},
$$
where $c,C>0$ are universal constants.
\label{lem_1233}
\end{lemma}

\begin{proof}
Let $m$ be an integer such that 
$$ R = m-k \in 2 \ZZ $$  and
$$
\left|m-\frac n5\right|\leq 1.
$$
Since $n \geq 20$, we have
\begin{equation} k+1\leq m\leq n-1,
\label{eq_1447} \end{equation}
and hence $R$ is a positive, even number. 
For a finite subset $Y \subseteq S$ define
\begin{equation}
\vphi(Y)=1_{\cA_k}(Y)\iota(Y).
\label{eq_1502} \end{equation}
Then $\vphi\geq 0$ while $\vphi(Y)=0$ whenever
the vectors in $Y$ are linearly dependent.
Let $\cE \subseteq \cF(S)$ be the collection of all linearly independent subsets of cardinality at most $n/10$. By Proposition \ref{prop_1445},
\begin{equation} \PP(L \cap S \in \cE)\geq 1 - C e^{-cn},
\label{eq_1458} \end{equation}
where we use that $\lambda\leq c_0n \leq n/30$.
For any finite set $B\subseteq S$, 
\begin{equation}
1_{\cA_k}(B)\leq 1_{\cE^c}(B) + \vphi(B)1_{\{|B|=k\}},
\label{eq_1935} \end{equation}
where $\cE^c = \cF(S) \setminus \cE$ is the complement of $\cE$. Applying (\ref{eq_1935}) with $B=L\cap S$ and using Proposition \ref{prop_1725}
with $R_2=R$, we get
\begin{equation}
1_{\cA_k}(L\cap S)
\leq
1_{\cE^c}(L \cap S)
+
\sum_{r=0}^{R}
(-1)^r
\sum_{X\in \mathcal T_{k,r}(L\cap S)}
\sum_{Y\in {X\choose k}}
\vphi(Y).
\label{eq_1949} \end{equation}
Recall from (\ref{eq_1424}) that $\mathcal T_{k,r}(L \cap S)$ consists of linearly independent
sets when $r$ is odd, and of sets of rank deficiency at most one when $r$ is
even. Since $0 \leq \vphi \leq 1$, we may bound the contribution 
of all sets $X$ with $\delta(X) = 1$ to the sum in (\ref{eq_1949})  as follows:
$$
\begin{aligned}
1_{\cA_k}(L\cap S)
& \leq \,
1_{\cE^c}(L \cap S)
+
\sum_{r=0}^{R}
(-1)^r
\sum_{X\in {L\cap S\choose k+r}}
\iota(X)
\sum_{Y\in {X\choose k}}
\vphi(Y) +
\sum_{r=0}^R  \sum_{\substack{ X\in {L\cap S\choose k+r} \\ \delta(X) = 1}}
{k+r \choose k}  \\
& \leq
1_{\cE^c}(L \cap S)
+
\sum_{r=0}^{R}
(-1)^r
\sum_{X\in {L\cap S\choose k+r}}
\iota(X)
\sum_{Y\in {X\choose k}}
\vphi(Y) +
\sum_{\ell=1}^{m}
{\ell\choose k}\rho_\ell^{\ell-1}(L\cap S).
\end{aligned}
$$
We may now take expectation and use Proposition \ref{prop_1832} in order to replace 
the random lattice~$L$ by the Poisson point process $P$. By using (\ref{eq_1458}) and Proposition \ref{prop_1832} we obtain
\begin{equation}
\PP(L\cap S\in \cA_k)
\leq \,
C e^{-cn}
+
\EE
\sum_{r=0}^{R}
(-1)^r
\sum_{X\in {P\cap S\choose k+r}}
\sum_{Y\in {X\choose k}}
\vphi(Y)
+
\EE
\sum_{\ell=1}^{m}
{\ell\choose k}\rho_\ell^{\ell-1}(L\cap S). \label{eq_1959}
\end{equation}
Here Proposition \ref{prop_1832} is applicable since $k+r  \leq k+R = m\leq n-1$ by (\ref{eq_1447}).

\medskip
With probability one, all subsets of $P\cap S$ of cardinality at most $m$ are
linearly independent. By (\ref{eq_1502}), almost surely  $\vphi(Y)=1_{\cA_k}(Y)$
for all $Y\in {P\cap S\choose k}$. 
Recall that $k+R = m$ and that $R$ is even. 
By the usual inclusion-exclusion inequality (\ref{eq_1503}) applied with $R_1 = R-1$,
\begin{equation}
\sum_{r=0}^{R}
(-1)^r
\sum_{X\in {P\cap S\choose k+r}}
\sum_{Y\in {X\choose k}}
\vphi(Y)
\, \leq \,
1_{\cA_k}(P\cap S)
+
\sum_{X\in {P\cap S\choose m}}
\sum_{Y\in {X\choose k}}
1_{\cA_k}(Y), \label{eq_2000}
\end{equation}
where the last term on the right-hand side is equal to the summand corresponding to $r = R$
on the left-hand side. Therefore, by taking expectation in (\ref{eq_2000}) and combining with (\ref{eq_1959}),
\begin{align} \nonumber
\PP(L\cap S\in \cA_k)
\leq \, &
C e^{-cn}
+
\PP(P\cap S\in \cA_k)
\\
&+
\EE
\sum_{X\in {P\cap S\choose m}}
\sum_{Y\in {X\choose k}}
1_{\cA_k}(Y)
+
\EE
\sum_{\ell=1}^{m}
{\ell\choose k}\rho_\ell^{\ell-1}(L\cap S). \label{eq_1507}
\end{align}
We need to bound the two terms in (\ref{eq_1507}). Begin by bounding the Poisson  term. Since $|P\cap S|$ is a Poisson random
variable with parameter $\lambda$,
\begin{align} \label{eq_1547}
\EE
\sum_{X\in {P\cap S\choose m}}
\sum_{Y\in {X\choose k}}
1_{\cA_k}(Y)
&\leq
{m\choose k}\EE {|P\cap S|\choose m} =
{m\choose k}\frac{\lambda^m}{m!}
\\
&\leq
2^m\left(\frac{e\lambda}{m}\right)^m
=
\left(\frac{2e\lambda}{m}\right)^m
\leq
C e^{-cn}. \nonumber
\end{align}
In the last line we used the bounds ${m \choose k} \leq 2^m$ and $m! \geq (m/e)^m$, as well as the facts that $\lambda\leq c_0 n$, $m\geq n/10$ and $c_0 \leq 1/200$.

\medskip
Next, let us bound the rank-deficiency-one  term. 
Recall from (\ref{eq_1447}) that $m \leq n-1$. 
By Schmidt's estimate
(\ref{eq_1451}), for $1\leq \ell\leq m$,
$$
\EE \rho_\ell^{\ell-1}(L\cap S)
\leq
\frac{\lambda^{\ell-1}}{(\ell-1)!}
\left[
3^\ell\left(\frac34\right)^{n/2}
+
5^\ell 2^{-n}
\right].
$$
Consequently, by using the inequality $\binom \ell k \leq 2^{\ell}$, 
\begin{align} \label{eq_1546}
\EE
\sum_{\ell=1}^{m}
{\ell\choose k}\rho_\ell^{\ell-1}(L\cap S)
&\leq
\sum_{\ell=1}^{\infty}
2^\ell
\frac{\lambda^{\ell-1}}{(\ell-1)!}
\left[
3^\ell\left(\frac34\right)^{n/2}
+
5^\ell 2^{-n}
\right]
\\ \nonumber
& = \left(\frac34\right)^{n/2} 6 e^{6 \lambda} + 10 e^{10 \lambda} 2^{-n}  \leq C e^{-cn}, \nonumber
\end{align}
where the last passage follows from the facts that  $\lambda\leq c_0n$ and $c_0 \leq 1/200$.
Substituting the  bounds (\ref{eq_1547}) and (\ref{eq_1546}) into (\ref{eq_1507}) we obtain
$$
\PP(L\cap S\in \cA_k)
\leq
\PP(P\cap S\in \cA_k)+C e^{-cn},
$$
completing the proof.
\end{proof}

\begin{proof}[Proof of Theorem \ref{thm1}]
We may assume that $n \geq 20$, since otherwise
the conclusion follows by increasing the universal constant $C$.
It suffices to prove the theorem under the assumption
$$
\Vol_n(S)=\lambda\leq c_0 n.
$$
By the definition of total variation distance,
\begin{align} \nonumber
d_{TV}(L\cap S,P\cap S)
& =
\sup_{\cA\subseteq \cF(S)}
\left| \,
\PP(L\cap S\in \cA)-\PP(P\cap S\in \cA) \,
\right| \\
& = \sup_{\cA\subseteq \cF(S)}
\left[ \, 
\PP(L\cap S\in \cA)-\PP(P\cap S\in \cA) \, 
\right], \label{eq_2018}
\end{align}
where the suprema run over all Borel subsets $\cA \subseteq \cF(S)$.
Note that in the second line we replaced the absolute value by brackets, which is legitimate
because replacing $\cA$ by $\cF(S)\setminus \cA$ changes the sign of the
expression inside the absolute value.
For $k\geq 0$, write $\cF_k(S)$ for the collection of all $k$-element subsets
of $S$. Set
$$
m=\lfloor n/10\rfloor+1
$$
and
$$
\cF_{\geq m}(S)=\bigcup_{\ell\geq m}\cF_\ell(S).
$$
Thus
$$
\cF(S)=
\left(\bigcup_{k=0}^{m-1}\cF_k(S)\right)
\cup
\cF_{\geq m}(S)
$$
is a disjoint union. 
Fix a Borel set $\cA\subseteq \cF(S)$ and write
$$
\cA_k=\cA\cap \cF_k(S),
\qquad
\cA_{\geq m}=\cA\cap \cF_{\geq m}(S).
$$
Then
\begin{align} \label{eq_2008}
&\PP(L\cap S\in \cA)-\PP(P\cap S\in \cA)
\\
&\qquad =
\left[
\PP(L\cap S\in \cA_{\geq m})
-
\PP(P\cap S\in \cA_{\geq m})
\right]
+
\sum_{k=0}^{m-1}
\left[
\PP(L\cap S\in \cA_k)
-
\PP(P\cap S\in \cA_k)
\right]. \nonumber 
\end{align}
For each $0\leq k\leq m-1$, Lemma \ref{lem_1233} yields
$$
\PP(L\cap S\in \cA_k)-\PP(P\cap S\in \cA_k)
\leq
C e^{-cn}.
$$
Since $m\leq n$, summing over $k=0,\ldots,m-1$ gives
\begin{equation}
\sum_{k=0}^{m-1}
\left[
\PP(L\cap S\in \cA_k)
-
\PP(P\cap S\in \cA_k)
\right]
\leq
C' e^{-c'n}.
\label{eq_2009} \end{equation}
It remains to bound the contribution of $\cA_{\geq m}$ to (\ref{eq_2008}). By Proposition
\ref{prop_1445}, with probability of at least $1-Ce^{-cn}$, the set $L\cap S$
has cardinality at most $n/10$. Since $m=\lfloor n/10\rfloor+1$, this implies
$$
\PP(L\cap S\in \cA_{\geq m})
\leq
\PP(|L\cap S|\geq m)
\leq
C e^{-cn}.
$$
Therefore
\begin{equation}
\PP(L\cap S\in \cA_{\geq m})
-
\PP(P\cap S\in \cA_{\geq m})
\leq
C e^{-cn}.
\label{eq_2010} \end{equation}
By combining (\ref{eq_2008}), (\ref{eq_2009}) and (\ref{eq_2010}) we conclude 
that for any Borel set $\cA\subseteq \cF(S)$,
$$
\PP(L\cap S\in \cA)-\PP(P\cap S\in \cA)
\leq
C e^{-c'n}.
$$ In view of (\ref{eq_2018}), this completes the proof. 
\end{proof}

\begin{remark} \begin{enumerate}
\item[(i)] Is it true that the universal constant $c$ from Theorem \ref{thm1} can be taken arbitrarily close to $1$?
This question is somewhat related to a remark in Venkatesh~\cite{ven}.

\item[(ii)] When $S \subseteq \RR^n$ is a Euclidean ball of volume $n / 5$, say, and $L \subseteq \RR^n$ is a random lattice of covolume one,
we currently do not know how to estimate
\begin{equation}  \left( \EE |L \cap S|^p \right)^{1/p} \label{eq_2015} \end{equation}
for all $1 \leq p < n$, up to a multiplicative universal constant. 
It seems that the bounds for (\ref{eq_2015}) that can be extracted from Schmidt \cite{schmidt2} could be 
 off by a factor of~$C^p$.
 The advantage of Schmidt's sieve over the usual inclusion-exclusion principle is that it allows us to avoid analyzing these high moments. 
\end{enumerate}
\end{remark}

\bigskip
\noindent School of Mathematical Sciences, Tel Aviv University, Tel Aviv 6997801, Israel; and
Department of Mathematics, Weizmann Institute of Science, Rehovot 7610001, Israel. \\
\textit{e-mail:} \texttt{klartagb@tau.ac.il}

\end{document}